%% file: driver.tex
\documentclass[11pt,letterpaper]{article}
\usepackage{amsmath,amsthm,amsfonts,color,hyperref,anysize,amssymb,color,graphicx,epstopdf,mathtools,mathrsfs}
\usepackage{float}
\usepackage[section]{placeins}
\usepackage[usenames,dvipsnames]{xcolor}
\usepackage[normalem]{ulem}
\usepackage{algorithmic}
\usepackage[margin=1.0in]{geometry}
\usepackage{multirow}
\usepackage{multicol}
\usepackage{subfiles}
\usepackage[inline]{enumitem}
\usepackage{algorithm}

\input{shared}
\input{header}

\input{header-extra}

\title{Parallel Stochastic Asynchronous Coordinate Descent: Tight Bounds on the Possible Parallelism}
\author{Yun Kuen Cheung
\thanks{Part of the work done while Yun Kuen Cheung held positions at Max Planck Institute for Informatics, Saarland Informatics Campus and at Singapore University of Technology and Design.
He was supported in part by Singapore NRF 2018 Fellowship NRF-NRFF2018-07 and MOE AcRF Tier 2 Grant 2016-T2-1-170.}
\\
		 Royal Holloway University of London
\and
		Richard Cole~~~~~~~~~~Yixin Tao
		\thanks{This work was supported in part by NSF Grants CCF-1909538 and CCF-1527568.}
\\
		Courant Institute, NYU
}
\date{}
\begin{document}\setlength{\parindent}{0.2in}
\maketitle
\input{abstract}

\newpage
\thispagestyle{plain}\setcounter{page}{1}

\input{intro-smooth}
\input{notation}
\input{asynchronous-schedule-simpler}
\input{result-simpler}

\input{easy-bound}

\input{high-prob-bound}

\input{alt-proof-all-times}

\input{discussion}
\input{acks}

 \bibliographystyle{plain}
 \bibliography{acd_references}

 \appendix

  \input{app}
\end{document}

%% file: shared.tex
\usepackage{lipsum}
\usepackage{amsfonts}
\usepackage{graphicx}
\usepackage{epstopdf}
\usepackage{algorithmic}
\ifpdf
  \DeclareGraphicsExtensions{.eps,.pdf,.png,.jpg}
\else
  \DeclareGraphicsExtensions{.eps}
\fi

%% file: header.tex
\usepackage{xcolor}
\usepackage{soul}

\usepackage{enumitem,amssymb,amsmath,multirow}
\newcommand{\timediff}{\tau}
\newcommand{\hide}[1]{}

\newcommand{\rr}{\mathbb{R}}

\newcommand{\ep}{\epsilon}

\newcommand{\ra}{\rightarrow}

\newcommand{\ceil}[1]{\left\lceil #1 \right\rceil}

\newcommand{\expect}[1]{\mathbb{E}\left[ #1 \right]}

\newtheorem{theorem}{Theorem}
\newtheorem{lemma}[theorem]{Lemma}
\newtheorem{corollary}[theorem]{Corollary}

\newtheorem{definition}{Definition}
\newtheorem{proposition}{Proposition}

\newcommand{\calC}{\mathcal{C}}

\newcommand{\calE}{\mathcal{E}}

\newcommand{\calO}{\mathcal{O}}

\newcommand{\bbb}{\mathbf{b}}

\newcommand{\bbe}{\mathbf{e}}

\newcommand{\bbx}{\mathbf{x}}
\newcommand{\bby}{\mathbf{y}}

\newcommand{\bbA}{\mathbf{A}}

\newcommand{\inner}[2]{\left\langle ~#1 ~,~ #2~\right\rangle}
\newcommand{\trans}{^{\mathsf{T}}}

\newcommand{\tx}{\tilde{x}}

%% file: header-extra.tex
\newcommand{\Ljj}{L_{jj}}
\newcommand{\Ljk}{L_{jk}}
\newcommand{\Lkj}{L_{kj}}

\newcommand{\pt}{x^t}
\newcommand{\qbar}{\overline{q}}
\newcommand{\ptone}{x^{t-1}}

\newcommand{\xc}{x^0}

\newcommand{\Lmax}{L_{\max}}
\newcommand{\Lres}{L_{\mathrm{res}}}
\newcommand{\Lresbar}{L_{\overline{\mathrm{res}}}}
\newcommand{\G}{\Gamma}

%% file: abstract.tex
\begin{abstract}
Several works have shown linear speedup is achieved by an asynchronous parallel implementation of stochastic coordinate descent so long as there is not too much parallelism. More specifically, it is known that if all updates are of similar duration, then linear speedup is possible with up to $\Theta(\Lmax\sqrt n/\Lresbar)$ processors, where $\Lmax$ and $\Lresbar$ are suitable Lipschitz parameters. This paper shows the bound is tight for almost all possible values of these parameters.
%
%The construction uses a variance suppressing biased random walk which might be of independent interest.
\end{abstract}

%% file: intro-smooth.tex
\section{Introduction}\label{sec::intro}

Very large scale optimization problems have arisen in many areas such as machine learning.
A natural approach for solving these huge problems is to employ parallel and more specifically asynchronous parallel algorithms.
As common wisdom suggests, when a small number of processors are used in these algorithms,
(linear) speedup can be achieved; but when
too many processors are involved and when they are not properly coordinated,
there may be undesirable outcomes.
(Recall that speedup is defined as the ratio of the algorithm's execution time on a single core and its parallel execution time. Linear speedup means the speedup is linear in the number of cores.)\footnote{In optimization, we compare the times of two executions that achieve a given level of accuracy.}
Typically, the bound on the number of processors is implicit and is expressed in terms of the maximum number $q$ of basic iterations,
namely single coordinate updates, that can overlap.
In many scenarios $q$ will be a small multiple of the number of processors or cores.
In many earlier works the notation $\tau$ was used instead of $q$.\footnote{We chose the notation $q$ to emphasize its likely similarity to $p$, the number of processors.}

%In the context of
This paper considers asynchronous implementations of stochastic coordinate descent (SCD) applied to smooth convex functions $f:\rr^n\ra\rr$.
% there are
Several recent works~\cite{LWRBS2015,LiuW2015,CCT2018} quantify how large $q$ can be in this context while guaranteeing
linear speedup.\footnote{Many of these results hold for \emph{composite} functions, i.e.,
functions of the form $f(\bbx) = g(\bbx) + \sum_{k=1}^n \Psi_k(x_k)$, where $g:\rr^n\ra\rr$ is a convex function with a continuous gradient,
and each $\Psi_k:\rr\ra\rr$ is a univariate convex function, but may be non-smooth.}
More precisely, they showed: if at any time at most $q\le \widetilde{q}$ updates can overlap, then linear speedup is guaranteed.
The goal in these works was to demonstrate as large a value of $\widetilde{q}$ as possible.
Note that these results provide \emph{lower} bounds on the actual value of $\widetilde{q}$.

The best existing lower bound is $\widetilde{q} = \Omega(\sqrt{n} \Lmax / \Lresbar)$, where $\Lmax$ and $\Lresbar$ are Lipschitz parameters
defined in Section~\ref{sec::notation}.\footnote{Here, we focus on the case where the step-size used in the asychronous SCD algorithm is 1/$\Lmax$;
we will discuss the cases with smaller step-sizes later in the introduction.}
Intuitively, one can view this as a lower bound on the possible parallelism supporting linear speedup; for if there are $p$ processors at hand, and if the durations of the updates vary by at most a factor of $d$,
then $q \le (d+1)(p-1)$, so achieving linear speedup with up to $q$ updates overlapping implies linear speedup occurs with $p=1+q/(d+1)$ processors.

We present the first work concerning the inverse problem: to identify a value $\qbar$,
such that if $q\ge\qbar$, then this can lead to an undesirable outcome.

\smallskip

\begin{center}
\begin{tabular}{ rl }
\textbf{\emph{Main result}:}~&  \emph{The lower bound of $\Omega(\sqrt{n} \Lmax / \Lresbar)$ is asymptotically tight.}  \\
& (An undesirable outcome can occur if $\qbar \ge \bar{c} \cdot \sqrt{n} \Lmax / \Lresbar$\\
& ~for a sufficiently large constant $\bar{c}$.)
\end{tabular}
\end{center}

\smallskip

We will present an adversarial family of functions with the following property: if $q$ exceeds $\Theta(\sqrt{n} \Lmax / \Lresbar)$ significantly,
then there is an asynchronous schedule for which with high probability
very rapid divergence occurs
% no significant progress toward convergence is made
for a very long time.
The function family uses a dimensionless parameter $\ep = \Theta(\frac{\Lresbar}{\Lmax\sqrt{n}})$, which is approximately the inverse of the possible parallelism.

We use the following function family, $f_\ep:\rr^n \ra \rr$:
\begin{equation}\label{eq:function-stall}
f_\ep(\bbx) = \frac{1 -\ep}{2}  \cdot \sum_{i=1}^n (x_i)^2 ~+~ \frac{\ep}{2} \cdot \left(\sum_{i=1}^n x_i\right)^2,
\end{equation}
for any $\ep$ satisfying $4/n \le \ep < 1$.
As we shall see, $\Lmax = 1$ and $\Lresbar = \Omega(\ep \sqrt n)$ for this function family;
thus the existing lower bound on the parallelism achieving linear speedup is $\Omega(\sqrt{n} \Lmax / \Lresbar) = \Omega(1/\ep)$.
To obtain bounds using arbitrary values of $\Lmax$ one can simply multiply $f_{\ep}$ by $\Lmax$, which also increases $\Lresbar$ by an $\Lmax$ factor.

Next, we discuss more precisely how we achieve this result.
Recall that the performance of a sequential SCD algorithm is expressed in terms of its convergence rate.
On strongly convex functions, it has a linear convergence rate, meaning that each update reduces the expected value of the difference
$f(\bbx) - f^*$ by at least an $(1 - \alpha/n)$ multiplicative factor, for some constant $\alpha > 0$,
where $f^*$ denotes the minimum value of the function. Consequently,
\begin{equation}\label{eq:upper-convergence-sequential}
\expect{f(\bbx^t) - f^*} ~~\le~~ \left( 1 - \frac \alpha n \right)^t \cdot \left(f(\bbx^0)- f^*\right).
\end{equation}
For our proposed function $f_\ep$, which is strongly convex, we will show that for a suitable initial point $\bbx^0$,
for some constant $\alpha' \ge\alpha$,
\begin{equation}\label{eq:lower-convergence-sequential}
\expect{f_\ep(\bbx^t) - f_\ep^*} ~~\ge~~ \left( 1 - \frac {\alpha'} n \right)^t \cdot \left(f_\ep(\bbx^0)- f_\ep^*\right),
\end{equation}
and hence sequential SCD achieves % we have
no more than a linear convergence rate in general.
For the function $f_\ep$, we will show that $\alpha =\tfrac 13$ and $\alpha' = 2$.

To achieve linear speedup with a parallel algorithm means that the same convergence rate holds,
up to constant factor, i.e., the $\alpha$ might be reduced by a constant factor $c \ge 1$, but no more:
{\begin{equation}
\expect{f(\bbx^t) - f^*} ~~\le~~ \left( 1 - \frac {\alpha}{cn} \right)^t \cdot \left(f(\bbx^0)- f^*\right),
\end{equation}
where $t$ is now the overall number of iterations performed by the various cores.\footnote{Liu et al.~\cite{LWRBS2015} and Liu and Wright~\cite{LiuW2015} call this near-linear speedup, reserving the term linear speedup for the case when $c=1$.}
This means that to guarantee a particular accuracy, the total number of iterations for a parallel execution is no more than a constant multiple of the number of iterations needed on a single core (so long as $\alpha \le n/2$).

Prior work has shown that linear speedup is achieved when $q \le \tilde{c} \sqrt{n} \Lmax / \Lresbar$ for some constant $\tilde{c}>0$.
To achieve our main result, we show that for the % proposed
function family $f_\ep$,
when $q \ge \bar{c} \sqrt{n} \Lmax / \Lresbar$ for some constant $\bar{c} > \tilde{c}$,
as an adversary, it is possible to pick asynchronous schedules such that for all $t\le n^{10}$ (or more generally, for any constant $\hat{c}\ge 1$,
for all $t \le n^{\hat{c}}$),
\[
\expect{f_\ep(\bbx^t) - f_\ep^*} ~\ge~ \Omega(4^{t/q})\cdot (f_\ep(\bbx^0) - f_\ep^*)
\]
for large enough $n$ (in general, when $n =\Omega(c^4)$). This indicates that when $q$ is too large, linear speedup cannot be achieved in worst-case scenarios.

The above upper bound on $q$ holds when the step-size is $1/\Lmax$. One might wonder what would happen if we reduced the step-size to $1/\G$ for some $\G\ge \Lmax$.
Would the permissible parallelism bound increase significantly, thereby improving the overall speedup?
In fact, we show that the upper bound on $q$ increases to at most $\calO(\G\sqrt{n}/\Lres)$, for any $\Lmax \le \G \le \calO(\Lres \sqrt n)$.
Since this increase is by a factor of at most $\calO(\G/\Lmax)$, but the step-size is reduced by a factor of $\Gamma/\Lmax$,
the overall speedup cannot improve by more than a constant factor.
This upper bound is also asymptotically tight, as there were matching lower bounds for these choices of step-sizes~\cite{CCT2018}.

\paragraph{Prior Work and Asynchrony Models}
First, note that the bound on $q$ is ensuring the asynchrony is bounded, and so we call it \emph{$q$-bounded asynchrony}.
Some requirement of this sort is unavoidable, otherwise there could be updates of arbitrarily long duration,
which, when they commit, could undo an arbitrary amount of progress.

In addition to the $q$-bounded asynchrony assumption,
we need to specify how the asynchronous environment affects the read operations.
There are two models concerning how coordinates are read in asynchronous environments,
namely ``consistent'' and ``inconsistent'' reads.
Our upper bound applies to both models. Next, we discuss their differences.

Liu et al.~\cite{LWRBS2015} gave the first bound on the parallel performance of asynchronous SCD on convex functions,
showing linear speedup when $q = O(\sqrt{n}\Lmax/\Lres)$ assuming a consistent read model,
where $\Lres$ is another Lipschitz parameter defined in Section~\ref{sec::notation}.
We note that $\Lres = \Lresbar$ for the function family $f_\ep$ we will be analyzing in this paper.
In fact, as we shall see, $\Lresbar$ is equal to $\Lres$ on all quadratic functions $f$, i.e.,
$f$ is of the form $\bbx\trans \bbA \bbx + \bbb\trans \bbx + \text{constant}$, where $\bbA$ is an $n\times n$ matrix, and $\bbb$ is an $n$-vector.
In the consistent read model, all the coordinate values a processor reads when performing a single update on one coordinate may be out of date,
but they must have been simultaneously current at some moment.
To make this more precise, we view the updates as committing at integer times $t=1,2,\ldots$,
and we write $\bbx^t$ to be the value of $\bbx$ after the update at time $t$.
The consistent reads model requires the vector of $\bbx$ values used by the time $t$ update to be of the form $\bbx^{t-\timediff}$ for some $\timediff \ge 1$.

Consistent reads create a substantial constraint on the asynchrony, and so subsequent works sought to avoid this assumption.
To this end, Liu and Wright~\cite{LiuW2015} proposed the inconsistent reads model.
Allowing inconsistent reads means that the $\tilde{\bbx}$ values used by the  time $t$ update can be any collection of the form $(x_1^{t-\timediff_1},\cdots,x_n^{t-\timediff_n})$,
where the $\timediff_j$'s can be distinct;
the $q$-bounded asynchrony assumption % translates to
implies that $1 \le \timediff_j \le q$ for each $j$.
Liu and Wright showed that linear speedup (including the more general case of composite functions)
can be achieved for $q = O(n^{1/4}\sqrt{\Lmax}/\sqrt{\Lres})$, i.e., the square root of the previous bound.
There remained several constraints on the possible asynchrony, in addition to the $q$-bounded asynchrony,
as pointed out by Mania et al.~\cite{MPPRRJ2017} and subsequently by Sun et al.~\cite{Sun2017}.
The latter works also gave analyses removing some or all of these constraints, but at the cost of reducing the bound on $q$.
Finally, Cheung, Cole and Tao~\cite{CCT2018} gave an analysis achieving linear speedup for $q = O(\sqrt{n}\Lmax/\Lresbar)$,
again for composite functions, with the only constraint being the $q$-bounded asynchrony.

%Consistent reads create a substantial constraint on the asynchrony, and so subsequent works sought to avoid this assumption.

In this paper we show the bounds in~\cite{LWRBS2015} and~\cite{CCT2018} are asymptotically tight for almost all possible values of
$\Lmax$, $\Lres$, and $\Lresbar$ for the function family \eqref{eq:function-stall}.

%% file: notation.tex
\section{Notation}
\label{sec::notation}

Let $\bbe_j$ denote the $n$-vector in which the $j$-th entry is 1 and every other entry is 0.

\begin{definition}\label{def:Lipschitz-parameters}
% The function $f$ is $L$-Lipschitz-smooth if for any $\bbx,\Delta \bbx\in\rr^n$, $\|\nabla f(\bbx+\Delta \bbx) - \nabla f(\bbx)\| ~\le~ L\cdot\|\Delta \bbx\|$.
%
For any coordinates $j,k$, the function $f$ is $\Ljk$-Lipschitz-smooth if for any $\bbx\in\rr^n$ and $r\in\rr$,
$|\nabla_k f(\bbx+r\cdot \bbe_j) - \nabla_k f(\bbx)| ~\le~ \Ljk\cdot |r|$;
it is $\Lres$-Lipschitz-smooth if, for all $j$, $\|\nabla f(\bbx+r\cdot \bbe_j) - \nabla f(\bbx)\| ~\le~ \Lres\cdot |r|$.
%Also, as is standard, $\Lj$ denotes $\Ljj$.
Finally, $\Lmax ~:=~ \max_{j} \Ljj$ and  $\Lresbar ~:=~ \max_k \left(\sum_{j=1}^n (\Lkj)^2\right)^{1/2}$.
\end{definition}

Observe that for the function $f_\ep$ we are considering,
\begin{equation}\label{eq:gradf}
\nabla_j f_\ep(\bbx) ~=~ (1 - \ep) \cdot x_j + \ep\cdot  \sum_{i=1}^n x_i.
\end{equation}
Consequently, $\Lmax = 1$, and $\Lresbar = \sqrt{1+(n-1)\ep^2} = \Theta(\ep \sqrt n)$, for $\ep = \Omega(1/\sqrt n)$.
% and $L = 1 + (n-1) \ep$.\footnote{$L$ equals the maximum eigenvalue of the Hessian for $f$;
% the eigenvector corresponding to the maximum eigenvalue is the all-one vector.}

\paragraph{The difference between $\Lres$ and $\Lresbar$}
In general, $\Lresbar\ge \Lres$.
$\Lres=\Lresbar$ when the rates of change of the gradient are constant, as for example in quadratic
functions such as $\bbx\trans \bbA \bbx + \bbb\trans \bbx +c$.
% All convex functions with Lipschitz bounds of which we are aware are of this type.
We refer the reader to~\cite{CCT2018} for a discussion of why
$\Lresbar$ is needed in general for the analysis in~\cite{CCT2018}.
\hide{
We need $\Lresbar$ because we do not make the Common Value assumption.
We use $\Lresbar$ to bound terms of the
form $\sum_j |\nabla_j f(y^j) - \nabla_j f(x^j)|^2$, where $|y^j_k - x^j_k| \le |\Delta_k|$,
and for all $h,i$, $|y^i_k -y^h_k|, |x^i_k -x^h_k| \le |\Delta_k|$,
whereas in the analyses with the Common Value assumption, the term being bounded is
$\sum_j |\nabla_j f(y) - \nabla_j f(x)|^2$, where $|y_k - x_k| \le |\Delta_k|$;
i.e., our bound is over a sum of gradient differences along the coordinate axes for
pairs of points which are all nearby, whereas the other sum is over gradient differences along the coordinate axes
for the same pair of nearby points.
Finally, if the convex function is $s$-sparse, meaning that each term $\nabla_k f(x)$ depends on at most $s$ variables,
then $\Lresbar \le \sqrt s \Lmax$.
}

Next, we define strong convexity.
\begin{definition}\label{def:strong-convexity}
Let $f: \rr^n \ra \rr$ be a convex function.
$f$ is strongly convex with parameter $\mu > 0$, if for all $x,y$,
$f(y) - f(x) \ge \inner{ \nabla f(x)}{y-x} + \frac 12 \mu ||y-x||^2$.
\end{definition}

A simple calculation shows that the parameter $\mu$ for our function $f_\ep$ has value $(1-\epsilon)$; see Lemma~\ref{lem::str-conv-bound-for-f} in~Appendix~\ref{app:missing}.

\paragraph{The update rule}
Recall that in a standard coordinate descent, be it sequential or parallel and synchronous, the update rule, applied to coordinate $j$,
first computes the accurate gradient $\nabla_j f(\bbx^{t-1})$,
and then performs the update given below.

\begin{equation}\label{eq:update-rule}
\pt_j ~\leftarrow~ \ptone_j - \frac{\nabla_j f(\bbx^{t-1})}{\G}
\end{equation}
and for all $k\neq j,~\pt_k \leftarrow \ptone_k$, where $\G\ge \Lmax$ is a parameter controlling the step size.

However, in an asynchronous environment,
an updating processor might retrieve outdated information $\tilde{\bbx}$ instead of $\bbx^{t-1}$,
so the gradient the processor computes will be $\nabla_j f(\tilde{\bbx})$,
instead of the accurate value $\nabla_j f(\bbx^{t-1})$. Hence the update rule is in the asynchronous environment is
\begin{equation}\label{eq:update-rule-simplify}
\pt_j ~\leftarrow~ \ptone_j  - \frac{\nabla_j f(\tilde{\bbx})}{\G}.
\end{equation}

\hide{
\paragraph{Scaling}
Sometimes, rather than have one value of $\G \ge \Lmax$ for all coordinates,
a value $\G_j \ge \Ljj$ is used for the $j$-th coordinate, typically with $\G_j/\Ljj$ being a fixed value for all $j$.
Alternatively and equivalently, one could rescale the coordinates so that each
$\Ljj =1$. The resulting $\Lmax = 1$ also, and we would now be using
a common value of $\G$ for the updates of all the coordinates.
Note that this may affect the values of $\Lres/\Lmax$ and $\Lresbar/\Lmax$.
}

We want to show that for any fixed constants $c_1,c_2 \ge 1$, %\st{constant $c\ge 1$,}
$f_\ep(\bbx^t) - f_\ep^*$
%\st{$\expect{f_\ep(\bbx^t) - f_\ep^*}$} is
is rapidly growing for $t \le q\cdot n^c_1$ with probability at least $1- 1/n^{c_2}$.
%\st{large for all $t \le n^c$.}

\subsection{The Stochastic Asynchronous Coordinate Descent (\textsf{SACD}) Algorithm}

The coordinate descent process starts at an initial point $\bbx^0 = (\xc_1,\xc_2,\cdots,\xc_n)$.
Multiple processors then iteratively update the coordinate values,
and for our
analysis we assume that at each time, there is exactly one coordinate value being updated, which we can do, as we are choosing the asynchronous schedule.

%\noindent\begin{minipage}{\textwidth}
\begin{algorithm}[H]
\caption{\textsf{SACD} Algorithm for Smooth Functions.}
\label{alg:SACD}
\textbf{Input: }The initial point $\bbx^0 = (\xc_1,\xc_2,\cdots,\xc_n)$.\\
~\\
\mbox{Multiple processors use a shared memory. Each processor iteratively repeats the}
\mbox{following four-step procedure, with no global coordination among them}:
\begin{algorithmic}
\STATE \textbf{Step 1:} Choose a coordinate $j \in \{1,2,\cdots,n\}$ uniformly at random.\label{alg-line-random}
\STATE \textbf{Step 2:} Retrieve coordinate values $\tx$ from the shared memory.\label{alg-line-retrieve}
\STATE \textbf{Step 3:} Compute the gradient $\nabla_j f(\tilde{\bbx})$.\label{alg-line-gradient}
%\STATE \textbf{Step 4:} Request a write lock on the memory that stores the value of\label{alg-line-lock-request}
%\STATE \hspace*{0.56in}the $j$-th coordinate.
%\STATE \textbf{Step 5:} Retrieve the most updated $j$-th coordinate value,
%then update
\STATE \textbf{Step 4:} Update coordinate $j$ using rule \eqref{eq:update-rule-simplify} by atomic addition. \label{alg-update}
%\STATE \textbf{Step 6:} Release the lock acquired in Step \ref{alg-line-lock-request}.
\end{algorithmic}
\end{algorithm}

\smallskip

%% file: asynchronous-schedule-simpler.tex
\subsection{Asynchrony Assumptions, Basic Set-up and Terminology Used in the Construction}\label{sect:setup}

In our construction, updates are made in phases. In each phase, $q$ updates are made in parallel by $q$ processors.
As in Step 1 of~Algorithm~\ref{alg:SACD}, each processor chooses a coordinate to update.
Since these processors are uncoordinated, they choose coordinates randomly and independently, and thus it is possible that some of them choose the same coordinate within a phase.
Then, for every coordinate, each processor retrieves the value that was up to date at the end of the previous phase.

Note that $\bbx^{q\ell}$ denote the up-to-date values
immediately after the $\ell$-th phase.
The starting point $\bbx^0$ can be viewed as the up-to-date values following
the (non-existent) $0$-th phase.
For each $\ell\ge 0$, in the $(\ell+1)$-st phase,
each of the $q$ processors performs the following sequence of steps:
\begin{itemize}
\item it picks a coordinate $k$ randomly and independently to update;
\item it retrieves $\bbx^{q\ell}$ and then computes $\nabla_k f(\bbx^{q\ell})$;
\item it increases the value of $x_k$ in the main memory by $-\nabla_k f(\bbx^{q\ell})/\Gamma$, using an atomic addition operation.
\end{itemize}

Note that in a single phase a particular coordinate might be chosen two or more times. If coordinate $k$ is chosen $b$ times in the $(\ell+1)$-st phase,
the value of $x_k$ after the $(\ell+1)$-st phase is $x_k^{q\ell} - b\cdot \nabla_k f(\bbx^{q\ell})/\Gamma$.

%% file: result-simpler.tex
\section{The Result}
\label{sec::result}

Cheung, Cole and Tao~\cite{CCT2018} showed the following upper bound on the performance
of asynchronous stochastic coordinate descent on strongly convex functions,
where $\bbx^*$ is a minimum point for the convex function.

\begin{theorem}
\label{thm::upper-bound-result}
i. Suppose the asynchronous updating is run for exactly $t$ updates.
Also suppose that  % $\Lmax \le \Gamma \le \Lresbar$ and
$q \le \min \left\{\tfrac{\sqrt{n}}{270},\tfrac{\Gamma \sqrt n} {270\Lresbar}\right\}$.
If $f$ is strongly convex with parameter $\mu$, then
\[
\expect{f(\bbx^t) -f(\bbx^*)} \le \left(1 - \frac 1 {3n} \cdot \frac {\mu}{\Gamma}\right)^t \left(f(\bbx^0) - f(\bbx^*)\right).
\]
ii. This result holds for all $q\le \tfrac{\Gamma \sqrt n} {12\Lres}$ if the asynchronous schedule obeys the Strong Common Value assumption.
\end{theorem}

The Strong Common Value assumption, specified in~\cite{CCT2018}, captures a substantial class of asynchronous schedules including the consistent read constraint from~\cite{LWRBS2015},
and all the schedules to which the analysis of Liu and Wright~\cite{LiuW2015} apply,
and more.
The complementary construction in this paper observes the consistent reads condition, which is the most restrictive of these constraints.
Consequently, our
goal is to show that for $q = \Omega(\tfrac{\G\sqrt{n}} {\Lres})$ convergence with linear speed-up is not guaranteed.
In fact, we will show that there is no convergence for most values of $q$ obeying this constraint.
To achieve this it suffices to show there is a single asynchronous schedule observing the
consistent read condition for which there is no convergence.

First, in Appendix~\ref{app:missing}, we will prove the following theorem, which shows
that the sequential stochastic coordinate descent on our function $f_\ep$ when starting at the point $\bbx^0=(-1,+1,-1,+1,\ldots)$
achieves a convergence
rate that is at most a constant factor faster than the convergence rate given in Theorem~\ref{thm::upper-bound-result}.
Recall that $\mu = 1-\ep$ for $f_\ep$.

\begin{proposition}
\label{thm::lower-bound-result}
Let $\bbx^0$ be the point with even-indexed coordinates equal to $+1$ and odd-indexed coordinates equal to $-1$.
Suppose the sequential stochastic coordinate descent is run for $t$ steps
on function $f_\ep$ starting at point $x^0$.
Then
\[
\expect{f_\ep(\bbx^t) - f_\ep(\bbx^*)} \ge \left(1 - \frac 2 {n}\cdot \frac {1 - \ep}{\Gamma}\right)^{t} \left(f_\ep(\bbx^0) - f_\ep(\bbx^*)\right).
\]
\end{proposition}

This result shows that for the function $f_\ep$, the speedup guaranteed by Theorem~\ref{thm::upper-bound-result} is indeed a linear speedup of the performance of the sequential SCD.

Now, we come to our main results.

Our analyses look at phases of $q$ successive updates when the coordinate descent in Algorithm~\ref{alg-update} is applied to the function $f \equiv \Lmax \cdot f_\ep$ in \eqref{eq:function-stall} with $\ep = \sqrt{\frac{\left( \frac{\Lres}{\Lmax} \right)^2-1}{n-1}}$.
The first result states that the expected value of $f(\bbx) - f^*$ grows by at least a factor of 4 from phase to phase.

\begin{theorem}
\label{thm::simplest-lower-bound}
Suppose that $\Lres \ge \Lmax \left(1 + 8/n\right)$ and $q \ge \frac{4\G \sqrt{n}}{\sqrt{\Lres^2 - \Lmax^2}}$. Then there is an asynchronous schedule
for which the coordinate descent diverges
when applied to the function $f$.
%$ \equiv \Lmax \cdot f_\ep$ in \eqref{eq:function-stall} with $\ep = \sqrt{\frac{\left( \frac{\Lres}{\Lmax} \right)^2-1}{n-1}}$.
Specifically,
for every $t=rq$, with $r\ge 1$ an integer,
\[
\expect{f(\bbx^{t}) - f^*} ~\ge~ 4^{(t/q)-1} \cdot (f(\bbx^0) - f^*).
\]

If $\Lres \ge 2 \Lmax$ (i.e.\ $\ep \ge \sqrt{3/(n-1)}$), the constraint on $q$ becomes $q \ge \tfrac{8\G\sqrt{n}}{\sqrt{3}\Lres}$.
\end{theorem}

When $\Lres \ge 2 \Lmax$, Theorem~\ref{thm::simplest-lower-bound} states that once $q$ exceeds the bound in Theorem~\ref{thm::upper-bound-result}(ii)
by a constant factor, in order to have any possibility of convergence, let alone linear convergence, $\G$ has to increase at the same rate as $q$. Note that the upper bound
on the rate of convergence decreases linearly with $\G$, so this says that increasing $q$, the
number of parallel updates, beyond this bound cannot increase the parallel runtime by
more than a constant factor at best.

However, one might wonder if this divergence in expectation is a low probability event.
Our next result shows that for most values of $q < n$ it is in fact a high probability event.

\begin{theorem}
\label{thm::high-prob-phase-lower-bound}
Let $c_1,c_2\ge 1$ be constants,  let $c=c_1+c_2$, and suppose that
\[ \frac n{e} \ge q \ge
\max\left\{ \tfrac {8\G}{\ep \Lmax}, \tfrac{10 (c+1)}{\ep}\tfrac{\log n}{\log \tfrac{n}{q}}\right\}.
\]
For each $c_1,c_2\ge 1$, with probability at least $1 - 1/n^{c_2}$, there is an asynchronous schedule for which the coordinate descent diverges for at least $q\cdot n^{c_1}$ updates
when applied to the function $f$.
Specifically, for every $t=rq$
with $1\le r\le n^{c_1}$ an integer,
\[
f(\bbx^{t}) - f^* ~\ge~ 4^{(t/q)-1} \cdot (f(\bbx^0) - f^*).
\]

If $\Lres\ge 2\Lmax$, the conditions
$\tfrac ne \ge q \ge  \max \left\{\tfrac {16\G \sqrt{n}}{\sqrt 3 \Lres}, \frac{200(c+1)\Lmax \sqrt n}{\sqrt 3 \Lres} \right\}$
and
$n\ge \left( \tfrac{544(c+1)}{\sqrt 3}\cdot \tfrac{\Lmax}{\Lres} \right)^{5/2}$
suffice.
\end{theorem}

Thus the comments in the paragraph after Theorem~\ref{thm::simplest-lower-bound} apply here too, so long as the stated bounds on $n$ hold.

Finally, one might wonder what happens to the value $f(\bbx) - f^*$ during a phase.
Could it be small at any time? As it happens, with our schedule the value oscillates a lot,
but the next and final result shows that with high probability it remains large at all times.

\begin{theorem}
\label{thm::all-times-lower-bound}
Let $c_1,c_2\ge 1$ be constants,  let $c=c_1+c_2$, and suppose that
$\Lres \ge 2\Lmax$, $n \ge 400c^2 (c+2)^2$, and
\[
\frac ne \ge q \ge \max \left\{ \frac{8\G\sqrt n}{\sqrt 3 \Lres} + \frac{4c\G}{\Lmax} ~,~ 20c(c+2) \right\}.
\]
For each $c_1,c_2\ge 1$, with probability at least $1 - 1/n^{c_2}$, there is an asynchronous schedule for which the coordinate descent diverges for at least $q\cdot n^{c_1}$ updates
when applied to the function $f$.
% $ \equiv \Lmax \cdot f_\ep$ in \eqref{eq:function-stall} with $\ep = \sqrt{\frac{\left( \frac{\Lres}{\Lmax} \right)^2-1}{n-1}}$.}
%With the assumptions and notation from Theorem~\ref{thm::high-prob-phase-lower-bound},
%\rjc{if in addition,  $n \ge \max\{(50ec+5e)^2, 25e^2(10c+1)^2 \left[\log \tfrac \G{\Lmax} + 1\right]^2\}$ and $10c+1 \le q \le  \tfrac n {5e^2\left[\log \tfrac \G{\Lmax} + 1\right]}$ }
Specifically, for every $1\let\le q \cdot n^{c_1}$,
\[
f(\bbx^{t}) - f^* ~\ge~ \frac{1}{n^2}\cdot 4^{\ceil{(t/q)}-1} \cdot (f(\bbx^0) - f^*).
\]
\end{theorem}

Once more, the comments in the paragraph after Theorem~\ref{thm::simplest-lower-bound} apply.

The three theorems incorporate general choices of $\G$, $c_1$ and $c_2$. In~Theorem~\ref{thm::high-prob-phase-lower-bound} and Theorem~\ref{thm::all-times-lower-bound},
if we pick $c_1,c_2$ to be large enough constants (for example, both are $10$ ---
which corresponds to a $1/n^{10}$ failure probability over the course of $q\cdot n^{10}$ updates),
%which correspond to events that look very bad already to most people,
the conditions on $n$ and $q$ reduce to $n\ge \Theta(1)$ and $n/e \ge q\ge \Theta(\G\sqrt n / \Lres)$, respectively.

%% file: easy-bound.tex
\section{Analysis}

We consider running SACD with $q$ processors on the function
\[
f(\bbx) = \Lmax \cdot \left[\frac{1 -\ep}{2}  \cdot \sum_{i=1}^n (x_i)^2 ~+~ \frac{\ep}{2} \cdot \left(\sum_{i=1}^n x_i\right)^2\right],
\]
for any $\ep$ satisfying $4/n \le \ep < 1$. %$1/\sqrt n \le \ep < 1$.
We prove each theorem in turn.

We choose the initial point to be the all ones point, $\bbx^0 = (+1,+1,\ldots,+1)$.
For $\ell\ge 0$, let $\bby^\ell$ denote the up-to-date values of the coordinates immediately after the $\ell$-th phase, i.e.~$\bby^\ell = \bbx^{q \ell}$.
Also, let
\[
G^\ell \triangleq \Big|\sum_{k=1}^n y^\ell_k \Big|\qquad\qquad\text{and}\qquad\qquad M^\ell \triangleq \max_{k=1,2,\cdots,n} |y^\ell_k|.
\]

\subsection{Proof of Theorem~\ref{thm::simplest-lower-bound}}
The key claim, implying exponential growth in the value of $f(\bbx)$ at the end
of each successive phase of $q$ updates, is given by the following lemma.

\begin{lemma}
\label{lem::expected-G-bound}
If $\ep \ge \frac 4n$ and $q \ge \frac {4\G}{\ep \Lmax}$ then for all $\ell$,
$\expect{G^\ell} \ge 2^\ell G^0$.
\end{lemma}
\begin{proof}
The proof is by induction. The claim clearly holds for $\ell=0$.

Now suppose the result holds for some $\ell\ge 0$. Without loss of generality, we assume $\sum_{j=1}^n y^\ell_j > 0$; the other case will be symmetric.
In the $(\ell+1)$-st phase, there are $q$ updates. Each update picks a coordinate $k$; by~\eqref{eq:gradf}, the update reduces its value by
\begin{align*}
\frac{\Lmax}{\Gamma} \cdot \bigg[ (1-\ep) y_k^\ell + \ep \sum_{j=1}^n y^\ell_j \bigg] &~=~ \frac{\Lmax}{\Gamma} \cdot \left[ (1-\ep) y_k^\ell + \ep G^\ell \right].
\end{align*}

The expected reduction due to $q$ updates is at least
\begin{align*}
\frac{q\Lmax}{\Gamma} \cdot \left[\ep G^{\ell} -\frac 1n(1-\ep)G^\ell \right]
& \ge \frac{3q\ep\Lmax}{4\Gamma}\cdot G^\ell~~~\text{(as $n\ep \ge 4$)}\\
& \ge 3 G^\ell~~~~~~~~~~~~~~\text{(as $q\ge \tfrac{4\G}{ \ep \Lmax}$).}
\end{align*}
Thus, $\expect{G^{\ell+1}|G^{\ell}} \ge \left|\expect{\sum_{k=1}^n y_k^{\ell+1} | G^{\ell}}\right|
\ge (3-1) G^\ell = 2G^\ell$;
% \ge \left(\frac{3q\ep\Lmax}{4\Gamma} - 1\right)\cdot G^\ell$, when $q \ge \frac{4\G}{3\ep \Lmax}$;
% and if $q\ge \frac{4\G}{ \ep \Lmax}$, then
therefore $\expect{G^{\ell+1}} \ge 2\cdot \expect{G^\ell}$, demonstrating the inductive claim.
\end{proof}

\begin{proof} [Proof of Theorem~\ref{thm::simplest-lower-bound}]
We begin by showing that $f(\bby^0) -f^* \le  \ep\cdot(G^0)^2$,
assuming that $\ep \ge 4/n$, which we justify in the next paragraph.
To see this, note that $G^0 =n$ and $f(\bby^0) -f^*= \frac 12(1 - \ep)n + \frac 12\ep n^2$.
As $\ep \ge 4/n$, we see that $f(\bby^0) -f^* \le \frac 18 \cdot \ep n^2+ \frac 12\ep n^2 \le \ep (G^0)^2$.
Next, immediately after Phase $\ell$, $f(\bby^\ell) -f^* \ge  \frac{\ep}{2}\cdot(G^\ell)^2$.
Then, by the Cauchy-Schwarz inequality and~Lemma~\ref{lem::expected-G-bound},
$\expect{f(\bby^\ell) -f^*} \ge \expect{\frac{\ep}{2}\cdot(G^\ell)^2} \ge \frac{\ep}{2}\cdot \expect{G^\ell}^2 \ge 4^{\ell-1}\ep \cdot (G^0)^2 \ge 4^{\ell-1} (f(\bby^0) -f^*)$.

It remains to show that $\ep \ge 4/n$.
Recall that
$\ep^2\Lmax^2 = (\Lres^2 - \Lmax^2)/(n-1)$.
Together, these imply $(\Lres^2 - \Lmax^2) \ge 16(n-1)/n^2\cdot \Lmax^2$.
It suffices to have $\Lres \ge \Lmax (1 +8/n)$, proving the theorem.
\end{proof}

%% file: high-prob-bound.tex
\subsection{Proof of Theorem~\ref{thm::high-prob-phase-lower-bound}}

The idea of the construction is to show that if we can bound $M^\ell$ by $\tfrac 14\ep G^\ell$,
then we can show the bound $G^\ell \ge 2^\ell G_0$ holds absolutely and not just in expectation. We will show that $M^\ell \le \tfrac14 \ep G^\ell$ with high probability.

Our construction uses a parameter $b$ which we will specify later.
We will need that in each
phase, each coordinate is chosen at most $b$ times.
In~Lemma~\ref{lem:not-too-many} and Corollary~\ref{cor:not-too-many} below, we show that this occurs with high probability when $b$ is suitably large.

\begin{lemma}\label{lem:not-too-many}
In one phase, the probability that each coordinate is chosen at most $b$ times by the $q$ processors is at least $1-n(eq/(b+1))^{b+1}/n^{b+1}$,
where $e = 2.71828\ldots$.
\end{lemma}

\begin{proof}
Let $k$ be one of the coordinates, and let $\calE(k)$ denote the event that coordinate $k$ is chosen more than $b$ times. Note that
\[
\calE(k) ~=~ \bigcup_{\substack{S:S\subset\{1,2,\cdots,q\}\\|S|=b+1}} \{\text{coordinate~}k\text{~is chosen by all the processors with labels in~}S\}.
\]
Thus, by the union bound, the probability that $\calE(k)$ holds is at most $\binom{q}{b+1} \cdot \left( \frac 1n \right)^{b+1}$,
which is at most $(eq/(b+1))^{b+1}/n^{b+1}$ by a well-known formula for bounding binomial coefficients.

By the union bound again, the probability that $\cup_{k=1}^n \calE(k)$ holds is at most $n(eq/(b+1))^{b+1}/n^{b+1}$. The lemma follows.
\end{proof}

\begin{corollary}\label{cor:not-too-many} \emph{[Events ${\cal E}_1$ and ${\cal E}_1'$]}
% Let $i\ge 1$. If $n^{1-1/2^{i-1}} < q \le n^{1-1/2^{i}}$, $b\ge 2$,
Suppose that $q \le \tfrac ne$,
$c_1, c_2\ge 1$, and $b \ge e -1$;
also, let $\Lambda = \tfrac{\log n}{\log (n/q)}$.
Then the probability that in each of the first $n^{c_1}$ phases each coordinate is chosen at most $b$ times by the $q$ processors
is at least $1-n^{c_1+1-(b+1)/\Lambda}$.
Furthermore,
\\
a. \emph{[Event ${\cal E}_1$]}
if $b \geq  (c_1 + c_2 + 1)/\Lambda$,
the probability is at least $1-1/n^{c_2}$; and
\\
b. \emph{[Event ${\cal E}_1'$]} if $n\ge 2$ and
$b \geq  (c_1 + c_2 + 2)/\Lambda$,
the probability is at least $1-1/n^{c_2+1} \ge 1 - 1/2n^{c_2}$.
\end{corollary}

Next, we show the inductive bounds on $M^\ell$ and $G^\ell$.
Recall that $c = c_1+c_2$.

\begin{lemma}
\label{lem::bounds-on-G}
For any fixed $c_1,c_2\ge 1$, conditioned on $\calE_1$,
if $\ep \ge \tfrac 4n$ and
$\tfrac{n}{e} \ge q \ge \max\left\{ \tfrac {8\G}{\ep \Lmax}, \tfrac{10 (c+1)}{\ep} \cdot \tfrac{\log n}{\log \tfrac{n}{q}} \right\}$,
then
for $\ell = 0,1,2,\ldots,n^{c_1}$, $G^\ell \ge 2^\ell G^0$ and
$M^\ell \le \ep G^\ell/4$.
\end{lemma}
\begin{proof}
Suppose the lemma holds for some $\ell\ge 0$, and without loss of generality, assume that $\sum_{j=1}^n y^\ell_j$ is positive; the other case will be symmetric.
As in Lemma~\ref{lem::expected-G-bound}, consider an update in the $(\ell+1)$-st phase that picks coordinate $k$; the update reduces its value by
\begin{align*}
\frac{\Lmax}{\Gamma} \cdot \left[ (1-\ep) y_k^\ell + \ep G^\ell \right]
~\ge~ \frac{\Lmax}{\Gamma} \cdot \left[ -(1-\ep) M^\ell + \ep G^\ell \right] ~\ge~ \frac{3\ep\Lmax}{4\G} G^\ell.
\end{align*}

Thus, immediately after the $(\ell+1)$-st phase,
$\sum_{j=1}^n y_j^{\ell+1} \le \sum_{j=1}^n y_j^\ell - q\cdot \frac{3\ep\Lmax}{4\G} G^\ell = \left( 1 - \frac{3 \ep q\Lmax}{4\G} \right) G^\ell.$
If $q\ge \frac{4\G}{ 3\ep\Lmax}$, then
\begin{equation}\label{eq:G-grows-quick}
G^{\ell+1} = \Big|\sum_{j=1}^n y_j^{\ell+1}\Big| ~\ge~ \left( \frac{3\ep q\Lmax}{4\G} - 1\right) G^\ell.
\end{equation}

On the other hand, for each coordinate $k$, if it is chosen by $b_k$ processors in the $(\ell+1)$-st phase, then the value of $y_k^{\ell+1}$ is
\[
y_k^\ell - b_k \cdot \frac{\Lmax}{\G} \cdot \left[(1-\ep) y_k^\ell + \ep G^\ell\right].
\]
For each $q$ we consider, we apply~Corollary~\ref{cor:not-too-many}(a).
 %{lem::bounds-on-G}(a)
Note that $\tfrac 1{\Lambda} =\tfrac{\log n}{\log \frac nq}$.
% by setting the parameter $i$ such that $n^{1-1/2^i} = q$, i.e. $2^i = \frac{\log n}{\log \frac nq}$.
Conditioned on $\calE_1$, $b_k \le \tfrac 1{\Lambda} (c+1)$.
% $b_k \le 2^i(c+1)$.
Also, since $|y_k^\ell|\le M^\ell \le \ep  G^\ell/4$
and $\Lmax \le \G$,
\begin{align}
|y_k^{\ell+1}| &\le \frac{\ep}{4} \cdot G^{\ell} + \frac 1{\Lambda} %{2^i}
(c + 1) \cdot \frac{\Lmax}{\G} \cdot \left( \frac{\ep }{4} \cdot G^{\ell} + \ep G^{\ell} \right)\nonumber\\
&~ = \left( \frac{1}{4} + \frac{5 \cdot  \frac 1{\Lambda} % \st{2^i}
(c+1)  \Lmax }{4\G}\right)\ep G^{\ell}.\label{eq:M-grows-slow}
\end{align}

Given~\eqref{eq:G-grows-quick} and \eqref{eq:M-grows-slow}, to guarantee that $G^{\ell+1}\ge 2G^\ell$, having $\frac{3\ep q\Lmax }{4\G} - 1 \ge 2$ suffices.
We satisfy this by imposing $q \ge \frac{8\G}{\ep \Lmax}$, or equivalently $\frac{\ep q\Lmax }{4\G}\ge 2$, which is slightly stronger than needed, but will help improve the next constraint on $q$.
To guarantee that $M^{\ell+1} \le \ep  G^{\ell+1}/4$, the following condition suffices:
\begin{equation}\label{eq:M-less-than-G}
\frac{\ep}{4}\left[\frac{3\ep q\Lmax}{4\G} - 1\right] G^\ell \ge  \left( \frac{1}{4} + \frac{5 \cdot  \frac 1\Lambda % \st{2^i}
(c+1)  \Lmax }{4\G}\right)\ep G^\ell.
\end{equation}

As $\frac{\ep q\Lmax }{4\G}\ge 2$, we see that $\tfrac{2\ep q \Lmax}{4\G} \ge \tfrac {5\cdot %2^i
(c+1) \Lmax}{\Lambda\G}$ suffices; i.e.\
$ q \ge  \tfrac{10 (c+1)}{\ep}\tfrac{\log n}{\log \frac{n}{q}}$ suffices.
\end{proof}

\begin{proof}[Proof of Theorem~\ref{thm::high-prob-phase-lower-bound}]
This theorem follows from Lemma~\ref{lem::bounds-on-G}, which requires that  Event ${\cal E}_1$ hold.
So the following conditions suffice: %need to hold:
\[
\ep \ge \frac 4n ~~\text{and}~~ \frac ne \ge q \ge \max\left\{ \frac {8\G}{\ep \Lmax}, \tfrac{10(c+1)}{\ep}\cdot\frac{\log n}{\log \frac{n}{q}}\right\}.
\]

If $\Lres\ge 2\Lmax$ (implying $\tfrac{1}{\ep} \le \tfrac{2\Lmax \sqrt n}{\sqrt 3 \Lres}$), the final condition becomes
\begin{align*}
q &\ge \max\left\{ \frac {16\G \sqrt{n}}{\sqrt 3 \Lres}, \frac{20(c+1)\sqrt{n}\Lmax}{\sqrt 3 \Lres} \cdot \frac{\log n}{\log \frac nq} \right\}.%\\
\end{align*}

We focus on the second constraint, namely $q \geq \frac{20(c+1)\sqrt{n}\Lmax}{\sqrt 3 \Lres} \cdot \frac{\log n}{\log \frac nq}$.
Let $\mathcal{C} = \frac{20(c+1)\sqrt{n}\Lmax}{\sqrt 3 \Lres}$.  The constraint becomes $q \geq \mathcal{C}\frac{\log n}{\log \frac nq}$, or $q\log \frac nq \ge \calC \log n$, where the RHS is independent of $q$. Note that $q \log \frac nq$ is an increasing function for $1 \leq q \leq \frac n{e}$.
Therefore, if suffices to seek a $\hat{q}\ge 1$ such that $\hat{q}\log  (n/{\hat{q}}) \ge \calC \log n$; then $q\ge \hat{q}$ implies that the second constraint holds for $q \le n/e$.

Consider $\hat{q} = 10\calC$. Then $\hat{q}\log (n/{\hat{q}}) \ge \calC \log n$ is equivalent to the inequality $9\calC \log n \ge 10\calC \log (10\calC)$, and hence to $n^9\ge (10 \calC)^{10}$.
Substituting for $\calC$, we obtain
\[
n^9 \ge \left( \frac{200(c+1)}{\sqrt 3}\cdot \frac{\Lmax}{\Lres}\cdot \sqrt n \right)^{10},~\text{or equivalently, }n \ge \left( \frac{200(c+1)}{\sqrt 3}\cdot \frac{\Lmax}{\Lres} \right)^{5/2}.
\]
We also need $\hat{q} = 10\calC \le n/e$;
$n \ge \left( \frac{544(c+1)}{\sqrt 3} \cdot \frac{\Lmax}{\Lres} \right)^2$ suffices.
\end{proof}

%% file: alt-proof-all-times.tex
\subsection{Proof of Theorem \ref{thm::all-times-lower-bound}}

\begin{lemma}
\label{lem::bound-on-M-improved}
Conditioned on Event $\calE_1'$ defined in~Corollary~\ref{cor:not-too-many}, if $\Lres \ge 2\Lmax$, $n \ge 400c^2 (c+2)^2$, and
\[
q \ge \max \left\{ \frac{8\G\sqrt n}{\sqrt 3 \Lres} + \frac{4c\G}{\Lmax} ~,~ 20c(c+2) \right\}
\]
then $M^{\ell}\le \frac{1}{4c}G^{\ell}$ and $G^{\ell} \ge 2^\ell G^0$, for $1\le \ell \le n^{c_1}$.
\end{lemma}
\begin{proof}
We follow the inductive argument in the proof of~Lemma~\ref{lem::bounds-on-G} closely. In the spirit of deriving~\eqref{eq:M-less-than-G}, we apply Corollary~\ref{cor:not-too-many}(b)
instead of Corollary~\ref{cor:not-too-many}(a), causing the $c+1$ term to be replaced by a $c+2$ term.
Then, it suffices that
\[
\frac{1}{4c}\left[\frac{3\ep q\Lmax}{4\G} - 1\right] G^\ell \ge  \left( \frac{1}{4} + \frac{5 \cdot  (c+2)  \Lmax }{4\G}\cdot \frac{\log n}{\log \frac{n}{q}}\right)\ep G^\ell.
\]
The above constraint is equivalent to $q \ge \frac{4\G}{3\ep \Lmax} + \frac{4\G c}{3\Lmax} + \frac{20c(c+2)}{3} \cdot \frac{\log n}{\log \frac nq}$,
so it suffices that $q \ge \max \left\{ \frac{4\G}{\Lmax} \left( \frac 1\ep + c \right) , 10c(c+2)\cdot \frac{\log n}{\log \frac nq} \right\}$.

We focus on the second constraint. As argued in the proof of~Theorem~\ref{thm::high-prob-phase-lower-bound}, it suffices to find a lower bound $\hat{q}$ on $q$, such that
$\hat{q} \log \frac n{\hat{q}} \ge \calC \log n$, where $\calC \triangleq 10c(c+2)$. $\hat{q} = 2\calC$ suffices if $n\ge (2\calC)^2$.

Since $q\ge \frac{4\G}{\ep \Lmax}$, following the derivation of~\eqref{eq:G-grows-quick} yields $G^{\ell+1} \ge 2G^{\ell}$.
\end{proof}

\begin{lemma}
\label{lem::Interval-I}
[Event ${\cal E}_2$]
Let ${\cal I} = [-\tfrac {1}{2n}G^\ell, \tfrac {1}{2n}G^\ell]$.
Suppose all the conditions imposed in Lemma~\ref{lem::bound-on-M-improved} hold.
Then, with probability at least $1 - 1/n^{c_2}$,
for $0 \le \ell < n^{c_1}$,
at every time during phase $\ell+1$,
at least one coordinate has a value outside $\cal I$.
\end{lemma}
\begin{proof}
We condition on Event $\calE_1'$, which by Corollary~\ref{cor:not-too-many},
occurs with probability at least $1- 1/2n^{c_2}$.

As before, without loss of generality, we assume that $\sum_i y^\ell_i > 0$.

We start by showing that at the start of the $(\ell+1)$-st phase at least $2c$ coordinates have values larger than $ \frac {1}{2n}G^\ell$.
Note that $\sum_{i:y^\ell_i\in {\cal I}} y^\ell_i \le \tfrac 12 G^\ell$. By~Lemma~\ref{lem::bound-on-M-improved},
$M^\ell \le \tfrac{1}{4c}G^\ell$. Thus, at the start of Phase
$\ell+1$, at least $\tfrac 12 G^\ell /\tfrac 1{4c} G^\ell = 2c$ coordinates have value greater than  $ \frac {1}{2n}G^\ell$, proving this claim.

Next, we observe that an update to a coordinate with value in $I$ changes
its value to be smaller than $-\frac {1}{2n}G^\ell$. For the largest value resulting from such
an update is $\left[1 -(1-\ep) \tfrac {\Lmax}{\G}\right] \tfrac  {1}{2n}G^\ell -
 \tfrac {\Lmax}{\G}  \ep G^\ell$.
By assumption, $\tfrac{\G}{\Lmax} \le \tfrac 14 \ep q \le \tfrac 14 \ep n$;
therefore $\tfrac{\ep \Lmax}{\G} \ge \tfrac 4n$.
We deduce the update value is at most
 $-\left( \tfrac 4n - \tfrac{1}{2n}\right) G^\ell$,
 which demonstrates the claim.

In order for all the coordinates to end up in $\cal I$ during Phase $\ell+1$,
we need that there be no coordinates less than $-\tfrac {1}{2n}G^\ell$ initially,
and that there be no updates to the coordinates in $\cal I$ until all the
other coordinates enter $\cal I$, assuming this is possible.
This requires some $k \ge 2c$ updates to these at least $2c$ coordinates.

The probability that these updates happen first is at most
\[
\frac{k!}{n^k} \le \frac{(2c)!}{n^{2c}} \le \left(\frac{2c}{n}\right)^{2c},
\]
and this is bounded by $1/(2n)^c$ if $n \ge 8c^2$.
Summed over all $n^{c_1}$ phases, this gives a total failure
probability of (significantly) less than $1/2n^{c_2}$.

Taking into account the $1/2n^{c_2}$ probability that Event $\calE_1'$ does not occur,
we see that the overall probability of Event ${\cal E}_2$ is at least $1 - 1/n^{c_2}$, as claimed.
\end{proof}

\begin{proof} [Proof of Theorem~\ref{thm::all-times-lower-bound}]
Conditioned on Event ${\cal E}_2$,
throughout phase $\ell+1$, at least one coordinate has
a value outside the interval $\cal I$ defined in Lemma~\ref{lem::Interval-I}.
Thus $f(\bbx) \ge \left[\tfrac {1}{2n}G^\ell\right]^2
\ge \tfrac{1}{4n^2}\cdot 4^\ell \left(G^0\right)^2
\ge \frac 14\cdot 4^\ell$.

Note that $f(\bbx^0) = \frac{1 - \ep}{2}\cdot n + \frac{\ep}{2} n^2 \le \ep n^2$ as, by assumption, $\ep > \frac 1n$.
We deduce that $f(\bbx) \ge  \frac {1}{4n^2}\cdot 4^\ell \cdot f(\bbx^0)$ throughout this phase.
\end{proof}

%% file: discussion.tex
\section{Discussion}

We have shown a tight asymptotic upper bound on the possible parallelism
for achieving linear speedup when using asynchronous coordinate descent for almost the whole range of $\frac{\Lresbar}{\Lmax}$.
This upper bound holds even for composite functions.
Furthermore, it holds even in the somewhat restrictive consistent read model,
and thus it holds for the inconsistent read model too.
%\st{However, it does require the inconsistent read model rather than the more restrictive consistent read model,
%so as to allow the adversary to have the flexibility to choose which updates are not read.}

%Our upper bound is a factor of $O(\ln n)$ off from the best currently known lower bound when
%$\Lres / \Lmax = \Omega(\frac{\sqrt n}{\ln n})$,
%i.e., when $\ep = \Omega(1/\log n)$ and our upper bound is $\Theta(\log n)$,
%rather than the currently best known lower bound of $\Theta(1/\ep)$.
%The primary reason is that in our construction, we need $q$ to be at least $\Theta(\log n)$
%to ensure that the flexibility of the adversary is maintained with high probability.

%% file: acks.tex
\subsection*{Acknowledgments}
We thank the referees %\st{of an earlier version of this paper} 
for their thoughtful and incisive comments. %\st{They were truly helpful.}

%% file: app.tex
\section{Missing Proofs}\label{app:missing}

\begin{lemma}
\label{lem::str-conv-bound-for-f}
The strong convexity parameter of $f_\ep$ is $(1-\ep)$.
\end{lemma}
\begin{proof}
\begin{align*}
&f_\ep(y) - f_\ep(x)   - \inner{\nabla f_\ep(x)} {y-x}\\
&\hspace*{0.6in} \ge \frac {1 - \ep} {2} \sum_i y_i^2 - x_i^2 + \frac {\ep} 2 \bigg[\big(\sum_i y_i\big)^2 - \big(\sum_i x_i\big)^2\bigg] \\
&\hspace*{1.0in}- (1 - \ep) \sum_i (x_i y_i - x_i^2) - \ep\bigg[\big(\sum_i x_i\big)\big(\sum_i y_i -\sum_i x_i\big)\bigg]\\
&\hspace*{0.6in}\ge \frac{1 - \ep}{2} \sum_i (y_i^2 + x_i^2 - 2x_i y_i)
+ \frac {\ep}{2} \bigg[ \big(\sum_i y_i\big)^2 + \big(\sum_i x_i\big)^2 - 2 \sum_i x_i \sum_i y_i\bigg] \\
&\hspace*{0.6in}\ge \frac{1 - \ep}{2}  \sum_i \big(y_i - x_i\big)^2 + \frac {\ep}{2}\bigg[\sum_i y_i - \sum_i x_i\bigg]^2 \\
&\hspace*{0.6in}\ge \frac{1 - \ep}{2} ||y-x||^2.
\end{align*}
\end{proof}

\paragraph{Proof of~Theorem~\ref{thm::lower-bound-result}}
Recall that $f_\ep(\bbx^*) = 0$. Also, recall that the SCD process starts at the all ones point.
%Consider the random variable $S(t) := \sum_{j=1}^n x_j^t$. Note that $S(0) = n$.
%Recall that if coordinate $j$ is chosen to be updated at time $t+1$,
%\[
%x_j^{t+1} - x_j^t ~=~ - \frac{\nabla_j f(\bbx^t)}{\G} ~=~ -\frac{1-\ep}{\G} \cdot x_j^t - \frac{\ep}{\G} \cdot S(t).
%\]
%Thus,
%\[
%\expect{S(t+1) - S(t) ~|~ S(t)} ~=~ -\frac{1-\ep}{\G n} \cdot S(t) - \frac{\ep}{\G}\cdot S(t) ~=~
%\]
Let $C_1$ denotes the even index coordinates, and $C_{-1}$ those of odd index.
Consider the following random variables:
\[
S_1(t) ~:=~ \sum_{j\in C_1} x_j^t~~~~~~S_{-1}(t) ~:=~ \sum_{j\in C_{-1}} (-x_j^t)~~~~~~S(t) ~:=~ S_1(t) + S_{-1}(t).
\]
Note that $S_1(0) = S_{-1}(0) = n/2$ and $S(0) = n$.

Next, we derive a recurrence which, conditioned on $S(t)$, computes the expected value of $S(t+1) - S(t)$.
Recall that if coordinate $j$ is chosen to be updated at time $t+1$, then
\[
x_j^{t+1} - x_j^t ~=~ - \frac{\nabla_j f(\bbx^t)}{\G} ~=~ -\frac{1-\ep}{\G} \cdot x_j^t - \frac{\ep}{\G} \cdot \left[ S_1(t) - S_{-1}(t) \right].
\]
Thus,
\begin{align*}
\expect{S(t+1) - S(t)~\Big|~S(t)}
&~=~ \frac 1n \left[ \sum_{j\in C_1} \left( -\frac{1-\ep}{\G} \cdot x_j^t - \frac{\ep}{\G} \cdot \left[ S_1(t) - S_{-1}(t) \right] \right) \right.\\
&\hspace*{0.5in}\left.~+~ \sum_{j\in C_{-1}} \left( \frac{1-\ep}{\G} \cdot x_j^t + \frac{\ep}{\G} \cdot \left[ S_1(t) - S_{-1}(t) \right] \right) \right]\\
&~=~ \frac 1n \left( -\frac{1-\ep}{\G} \cdot S_1(t) - \frac{1-\ep}{\G} \cdot S_{-1}(t) \right)\\
&~=~ -\frac{1-\ep}{n\G} S(t).
\end{align*}
The second equality above holds because $|C_1| = |C_{-1}|$, leading to cancellation of the terms $\pm \frac{\ep}{\G} \cdot \left[ S_1(t) - S_{-1}(t) \right]$.
Thus, $\expect{S(t+1)~|~S(t)} = S(t)\cdot \left( 1- \frac{1-\ep}{n\G} \right)$. Iterating this recurrence yields
\begin{equation}\label{eq:expect-St}
\expect{S(t)} = n\cdot \left( 1- \frac{1-\ep}{n\G} \right)^t.
\end{equation}

Next, observe that for any fixed $S_1(t),S_{-1}(t)$, by the Power-Mean Inequality,
$\sum_{j\in C_1} (x_j^t)^2 \ge \frac{2|S_1(t)|^2}{n}$ and $\sum_{j\in C_{-1}} (x_j^t)^2 \ge \frac{2|S_{-1}(t)|^2}{n}$.
On the other hand, for any fixed $S(t)$, which equals \st{to} $S_1(t) + S_{-1}(t)$, the sum $\frac{|S_1(t)|^2}{n} + \frac{|S_{-1}(t)|^2}{n}$
is minimized when $S_1(t) = S_{-1}(t) = S(t)/2$. Thus,
\begin{align*}
\expect{f_\ep(\bbx^t)} ~\ge~ \expect{\frac{1-\ep}2\sum_{j=1}^n (x_j^t)^2} &~\ge~ (1-\ep)\cdot \expect{\frac{|S_1(t)|^2}{n} + \frac{|S_{-1}(t)|^2}{n}} \\
&~\ge~ \frac{1-\ep}{2n}\cdot \expect{S(t)^2}.
\end{align*}
Finally, we complete the proof by using the Cauchy-Schwarz inequality and \eqref{eq:expect-St}:
\begin{align*}
\frac{1-\ep}{2n}\cdot \expect{S(t)^2}~\ge~ \frac{1-\ep}{2n}\cdot  \expect{S(t)}^2
&~=~ \frac{(1-\ep)n}{2}\cdot \left( 1- \frac{1-\ep}{n\G} \right)^{2t}\\
&~\ge~ \frac{(1-\ep)n}{2}\cdot \left( 1- \frac{2(1-\ep)}{n\G} \right)^{t} \\
&~=~ f_\ep(\bbx^0) \cdot \left( 1- \frac{2(1-\ep)}{n\G} \right)^t.
\end{align*}